\documentclass[10pt, reqno]{amsart}
\usepackage{amssymb}
\usepackage{amsmath}
\usepackage{amscd}
\usepackage{mathrsfs}
\usepackage{float}
\usepackage{graphicx}
\usepackage{caption}
\usepackage{subcaption}
\usepackage{amsfonts}
\usepackage{color}
\usepackage{pb-diagram}
\usepackage[utf8]{inputenc}
\usepackage[T1]{fontenc}
\usepackage{framed}
\usepackage{cancel}
\usepackage{hyperref}
\usepackage{times}

\newcommand{\B}{\mathtt{B}}
\newcommand{\Y}{{\rm Y}_{d,n}^{\mathtt B}}

\newcommand{\s}{\mathtt s}
\newcommand{\R}{\mathtt r}
\newcommand{\T}{\mathcal{T}_{ST}}
\newcommand{\TB}{TB_n^{\mathtt{B}}}
\newcommand{\U}{\mathsf{u}}
\newcommand{\f}{\mathcal{F}}
\newcommand{\V}{\mathsf{v}}

\newcommand{\x}{\mathsf{x}}
\newcommand{\tr}{\mathtt{tr}_{n}}
\newcommand{\y}{\mathsf{y}}
\newcommand{\z}{\mathsf{z}}

\newcommand{\w}{\mathsf{w}}
\newcommand{\E}{\mathcal{E}_n^{\mathtt{B}}}

\newtheorem{theorem}{Theorem}

\newtheorem{proposition}{Proposition}
\newtheorem{definition}{Definition}
\newtheorem{remark}{Remark}
\newtheorem{example}{Example}

\begin{document}

\title{Tied links in the solid torus}

\author{Marcelo Flores\footnotesize{$^1$}}
\thanks{$^1$ This project was partially supported by FONDECYT 11170305.}
\address{$^{1}$ Instituto de Matem\'{a}ticas \\ Universidad de Valpara\'{i}so \\ Gran Breta\~{n}a 1091, Valpara\'{i}so, Chile}
\email{marcelo.flores@uv.cl}


\maketitle
\begin{abstract}
  We introduce the concept of tied links in the solid torus, which generalize naturally the concept of tied links in $S^3$ previously introduced by Aicardi and Juyumaya. We also define an invariant of these tied links by using skein relations, and subsequently we recover this invariant by using Jones' method over the bt-algebra of type $\B$ and the Markov trace defined on this.
\end{abstract}

\section{Introduction}
In \cite{jo83, jo84}, Jones constructs the famouus Jones polynopmial by using the Markov trace
function on the tower of classical Temperley-Lieb algebras. These algebras can be regarded as quotients of the associated Hecke algebras of type A. Subsequently, in \cite{jo}, he applies this procedure to Hecke algebras of type $\mathtt{A}$, obtaining as result the Homplypt polynomial, which had been defined previously in \cite{Homfly} by using skein relations. These procedure led to the idea of knot algebras, that are towers of algebras which support a Markov trace which may be rescaled and allow thus the construction new of invariants for knotted objects. The Hecke algebra, the Temperley-Lieb algebra and the BMW algebra are the most well-known examples of knot algebras. \smallbreak

The Yokonuma--Hecke algebra, which was originally introduced by T. Yokonuma \cite{yo} in the context of Chavalley gruops, is other significant example of knot algebra. Indeed, in \cite{juJKTR}, Juyumaya proves that the tower of Yokonuma--Hecke algebras support a unique Markov trace. Subsequently, by using Jones' method invariants for: framed links \cite{jula5}, classical links \cite{jula4} and singular links \cite{jula3} were constructed. Moreover, recently it was proved that the invariants for classical links constructed in \cite{jula4} are not topologically equivalent either to the Homflypt polynomial or to the Kauffman polynomial, see \cite{chjukala}.\smallbreak%This algebra
In \cite{aiju}, Aicardi and Juyumaya introduces the algebra of braids and ties (or bt--algebra), denoted by $\mathcal{E}_n$, that is also a knot algebra. The term  \lq braids and ties\rq\ refers to the generators of this algebra, which have a diagrammatical interpretation in terms of braids and ties (see \cite[Section 6]{aijuMMJ}). This algebra is defined by abstractly considering it as a certain subalgebra of the Yokonuma--Hecke algebra ${\rm Y}_{d,n}:= {\rm Y}_{d,n}(\U)$. Subsequently, in \cite{aiju2}, a Markov trace for $\mathcal{E}_n$ is constructed by implementing the method of relative traces (see \cite{aijuMMJ}, cf. \cite{chpoIMRN,fjl,isog}). Then using Jones' method \cite{jo} over this trace, the invariant for classical knots $\overline{\Delta}(\U,\mathsf{A},\mathsf{B})$ and the invariant for singular knots $\overline{\Gamma}(\U,\mathsf{A},\mathsf{B})$ are define. It is worth noting that, for links, the invariant $\overline{\Delta}$ is more powerful than the Homflypt polynomial (see \cite[Addendum]{aijuMMJ}).  In \cite{aiju2}, the authors introduce the concept of tied links in $S^3$, which generalize classical links. These objects are the closure of tied braids that are come from the diagrammatical interpretation of the defining generators of the bt--algebra. Subsequently, they construct the invariant $\mathcal{F}$ for these new objects via skein relations. Finally, they prove an analogue of the Alexander and Markov theorem for tied links and recover the invariant $\mathcal{F}$ by applying Jones' method to the bt--algebra together with the Markov trace defined in \cite[Section 4]{aijuMMJ}.  \smallbreak

All the results that are mentioned above are related to Coxeter groups of type $\mathtt{A}$. However, there has been a growing
interest also in knot algebras related to Coxeter systems of type $\mathtt{B}$. Indeed, the affine and
cyclotomic Yokonuma-Hecke algebra is introduced in \cite{chpoIMRN}, and recently the first author together Juyumaya and
Lambopoulou \cite{fjl} introduced the algebra $\Y(\U,\V)$, that can be regarded as an analogue of the Yokonuma--Hecke algebra in the context of Coxeter groups of type $\mathtt{B}$. This algebra supports a Markov trace (see \cite[Theorem 3]{fjl}), consequently invariants for framed knots and links in the solid torus are obtained by applying Jones' method. In order to generalize the classical bt--algebra, in \cite{fl}, a braids and ties algebra of type $\mathtt{B}$ is introduced, denoted by $\E:=\E(\U,\V)$, for $n\geq 1$. This algebra is defined in analogy to the construction of the bt-algebra of type $\mathtt{A}$, that is, it is obtained by considering it abstractly as a certain subalgebra of $\Y(\U,\V)$. We further prove that $\E$ supports a Markov trace, and using this trace as the main ingredient in Jones' method, we then define an invariant of classical links in the solid torus $ST$.\smallbreak

Then, it is natural try to define a analogue to the concept of tied links, though in the context of Coxeter groups of type $\mathtt{B}$. Thus, the purpose of this article is to introduce the concept of tied links in the solid torus, which generalize naturally the concept of tied links in $S^3$ previously introduced by Aicardi and Juyumaya. We do so by considering the diagrammatic interpretation of the bt--algebra of type $\mathtt{B}$ \cite[Section 3.1]{fl}. We then define the invariant $\mathcal{F}_{\B}$ for tied links in $ST$ via skein relations. Finally, we prove analogues of the Alexander and Markov theorems in order to recover the invariant $\mathcal{F}_{\B}$ by applying Jones' method to the algebra $\E$.\smallbreak

The article is organized as follows. In Section~\ref{sec2}, we provide the notation and necessary results. In Section~\ref{sec3}, using an analogy to the classical case, we introduce the concept of tied links in $ST$. In Section~\ref{sec4} we define the invariant $\mathcal{F}_{\B}$, which coincide with $\mathcal{F}$ considering tied links in $S^3$ as affine tied links in $ST$ (see Remark~\ref{restriction2}). We prove that this invariant is unique by defining it via skein relations (see Theorem~\ref{skeinth}). In Section~\ref{sec5}, we introduce the Tied braid monoid of type $\mathtt{B}$, denoted $TB_n^{\B}$, which contains the monoid $TB_n$ originally defined in \cite[Section 3.1]{aiju2}. The monoid $TB_n^{\B}$ plays the role that $B_n$ does in the context of classical links in $S^3$. That is, we use this monoid to prove in Section~\ref{sec6} analogues of Alexander and Markov theorem for tied links. In Section~\ref{sec7}, using the natural the representation of $TB_n^{\B}$ in $\E$ (see Proposition~\ref{rep}), we define an invariant by using Jones' method on the Markov trace defined in \cite[Section 5]{fl}. Finally, we prove that the invariant obtained by this procedure is equivalent to the invariant $\mathcal{F}_{\B}$ from Section~\ref{sec4}.

%In recent years, several knot algebras have been recently defined with the aim to obtain new invariants of knotted objects by using the framization technique \cite{go,fjl, chpoIMRN}. This procedure consists of adding a set of extra generators, called framing generators, to a known knot algebra a together with certain relations between this new generators and original ones \cite{jula7} in order to obtain invariants of framed knots. The concept of framization is motivated by

\section{Preliminaries}\label{sec2}

In this section, we recall some basic results to be used. We begin recalling some basic notions about braids of type $\B$ and links in the solid torus.

\subsection{Gruops of type $\B_n$}\label{braidsB}
Set $n\geq 1$. We denote the Coxeter group of type ${\mathtt  B}_n$ by $W_n$. This group is the finite Coxeter group associated to the Dynkin diagram
\begin{center}
\setlength\unitlength{0.2ex}
\begin{picture}(350,40)
\put(82,20){$\R_1$}
\put(120,20){$\s_{1}$}
\put(200,20){$\s_{n-2}$}
\put(240,20){$\s_{n-1}$}

\put(85,10){\circle{5}}
\put(87.5,11){\line(1,0){35}}
\put(87.5,9){\line(1,0){35}}
\put(125,10){\circle{5}}
\put(127.5,10){\line(1,0){10}}

\put(145,10){\circle*{2}}
\put(165,10){\circle*{2}}
\put(185,10){\circle*{2}}

\put(205,10){\circle{5}}
\put(207.5,10){\line(1,0){35}}
\put(245,10){\circle{5}}
\put(192.5,10){\line(1,0){10}}

%\put(100,7){$<$}

\end{picture}
\end{center}

The group $W_n$ can be realized as a subgroup of the permutation group of the set  $X_n:=\{-n, \ldots ,-1,1,$ $\ldots, n\}$. More specifically, the elements of $W_n$ are the permutations $w$ such  that $w(-m) = - w(m)$, for all $m \in X_n$.\smallbreak
Note that there is a natural projection $\tau:W_n\rightarrow S_n$, defined by $\R_1\mapsto 1$ and $\s_i\mapsto s_i$, where $s_i$ are the Coxeter generators of the symmetric group $S_n$.

The corresponding \textit{braid group of type} ${\mathtt  B}_n$ associated to $W_n$ is defined as the group $\widetilde{W}_n$ generated  by $\rho_1 , \sigma_1 ,\ldots, $ $\sigma_{n-1}$ subject to the following relations:

 \begin{equation}\label{braidB}
\begin{array}{rcll}
 \sigma_i \sigma_j & =  & \sigma_j \sigma_i, & \text{ for} \quad \vert i-j\vert >1,\\
  \sigma_i \sigma_j \sigma_i & = & \sigma_j \sigma_i \sigma_j, & \text{ for} \quad \vert i-j\vert = 1,\\
   \rho_1\sigma_i&=&\sigma_i\rho_1, &\text{ for}\quad i>1,\\
\rho_1 \sigma_1 \rho_1\sigma_1 & = & \sigma_1 \rho_1 \sigma_1\rho_1. &
  \end{array}
\end{equation}

Geometrically, braids of type ${\mathtt  B}_n$ can be regarded as classical braids of type ${\mathtt  A}_{n}$ with $n+1$ strands, such that the first strand is identically fixed. This strand is called `the fixed strand'. The 2nd, \ldots, $(n+1)$st strands are renamed from 1 to $n$ and are called `the moving strands'. The `loop' generator $\rho_1$ stands for the looping of the first moving strand around the fixed strand in the right-handed sense, see \cite{la1,la2}. Figure \ref{bbraid} illustrates a braid of type ${\mathtt  B}_6$.

\begin{figure}[h!]
\centering
\begin{minipage}{.4\textwidth}
  \centering
  \includegraphics{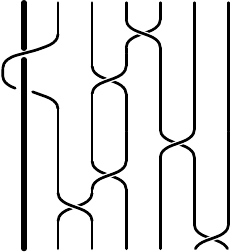}
  \captionof{figure}{A braid of type ${\mathtt  B}_6$.}
  \label{bbraid}
\end{minipage}%
\begin{minipage}{.6\textwidth}
  \centering
  \includegraphics{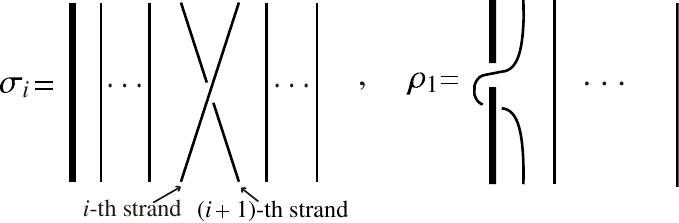}
  \captionof{figure}{Generators of $\widetilde{W}_n$.}
  \label{gentypeB}
\end{minipage}
\end{figure}

\subsection{Knots and links in the solid torus}\label{knotsB} It is well known that the solid torus $ST$ may be regarded as the complement in $S^3$ of another solid torus $\hat{I}$, i.e. $ST=S^3/\hat{I}$. So links in $ST$ can be regarded as \emph{mixed links} in $S^3$ containing the complementary solid torus. Therefore, any link $L$ in $ST$ is represented by a mixed link $\hat{I}\cup L'$ in  $S^3$, consisting of a standard link $L'$ in $S^3$, which is linked in some way with the fixed complementary torus part $\hat{I}$ (see Figure~\ref{linkST}). Consequently, a mixed link diagram is the projection of $\hat{I}\cup L'$ on the plane of the projection of $\hat{I}$. These facts also stand for oriented links in $ST$. \smallbreak

Thus, from now on, any oriented link $L$ in $ST$ with $n$ components will be seen as an oriented mixed link in $S^3$ with $n+1$ components. The component that represents the complementary solid torus, which is fixed and unknotted, will be called the \emph{fixed component} of $L$. The others $n$ components will be called the standard components of $L$.

\begin{figure}[h!]
\begin{center}
  \includegraphics{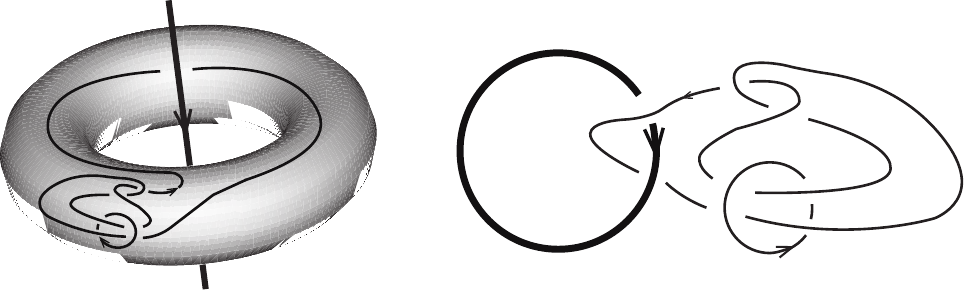}
	\caption{A link in the solid torus and its corresponding mixed link diagram.}\label{linkST}
	\end{center}
 \end{figure}
Thus, this mixed link diagram has crossings between standard components, called \emph{standard crossings} or simply crossings, and eventually it also has some loopings between the standard components and the fixed component, which are called loops (see Figure~\ref{crossings}).

\begin{figure}[h!]
\begin{center}
  \includegraphics{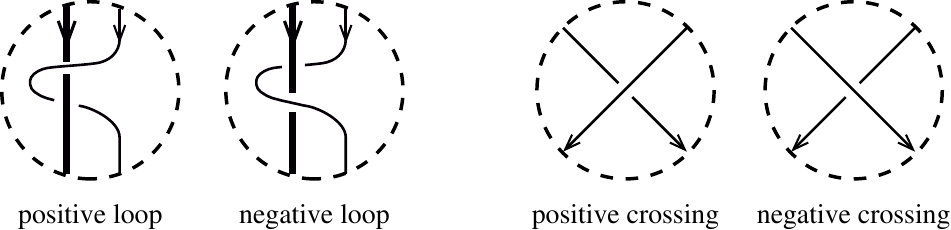}
	\caption{Standard crossing and loops.}\label{crossings}
	\end{center}
 \end{figure}

\begin{remark}\label{affinelinks}
  A link $L$ in $ST$ is called \textit{affine} if it lies in 3--ball in $ST$. Or in other words, the link $L$ does not have loops around the fixed component. Thus, classical links in $S^3$ can be regarded as an affine links in the solid torus.
\end{remark}

Two links in $ST$ are isotopic if and only if any two corresponding mixed links diagrams in $S^3$ differ by a planar isotopy and a finite sequence of mixed moves (see Figure~\ref{ReidM}) together with the three Reidermeister moves for the standard part of the link (see \cite{la1} for details).\\

\begin{figure}[h!]
\begin{center}
  \includegraphics{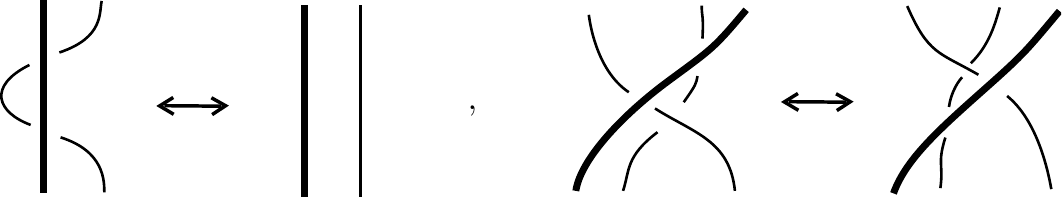}
	\caption{Reidemiester moves between mixed components.}\label{ReidM}
	\end{center}
 \end{figure}

Observe that Reidemeister moves in Figure~\ref{ReidM} imply the following move

  \begin{figure}[h!]
\begin{center}
  \includegraphics{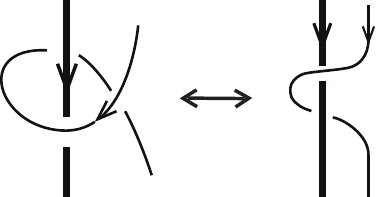}
		\end{center}
 \end{figure}
\noindent and the analogue one for the negative loop.\\

The closure of a braid $\alpha$ in the group $\widetilde{W}_n$ is defined by joining with simple (unknotted and unlinked) arcs its corresponding endpoints, and it is denoted by $\widehat{\alpha}$. The result of closure, $\widehat{\alpha}$, is a link in the solid torus. Thus, we have the following analogues of Alexander and Markov theorems for links in ST (see \cite{la1} for details).

\begin{theorem}\label{AlexanderST}
  Any oriented link in $ST$ is isotopic to a closure of a braid of type $\B$.
\end{theorem}

\begin{theorem}\label{MarkovST}
 Isotopy classes of oriented links in $ST$ are in bijection with equivalence classes of $\bigcup_n \widetilde{W}_n$, the inductive limit of braid groups of type $\mathtt{B}$, respect to the equivalence relation $\sim_{\mathtt B}$:
\begin{itemize}
  \item[(i)] $\alpha\beta \sim_{\mathtt B} \beta\alpha $,
  \item[(ii)] $\alpha\sim_{\mathtt B} \alpha\sigma_n$ and $\alpha\sim_{\mathtt B} \alpha\sigma_n^{-1}$,
\end{itemize}
for all $\alpha, \beta \in W_n$.
\end{theorem}

%From now on, by $L$ a link with $n$ components in $ST$ we will mean the corresponding mixed link in $S^3$ with $n$ standard components together the fixed one.

\section{Tied Links in the solid torus}\label{sec3}
In this section, we introduce the concepts of tied links in the solid torus and their diagrams. Indeed, a tied link in ST is simply a standard link in ST whose set of components are related in some way. We use ties as a formalism to indicate that two components are related. The ties will be drawn as a wavy line between two such components. These new knotted objects naturally generalize links in ST and classical tied links in $S^3$ (see \cite{aiju2}).

\begin{definition}\rm
   A tied (oriented) link in $ST$ with $n$ components is a pair $L(I):=(L,I)$, where $L$ is a link in $ST$ and $I$ is a collection of unordered pairs of points $(p_i, p_j)$ of $L$ (points in the fixed component are allowed). We called $I$ the set of ties. Thus, a pair $(p_i, p_j)\in I$ is represented as an wavy arc called tie that connects the points $p_i$ and $p_j$, which may belong to different components or to the same one. Ties they are not embedded arcs, they are just a notational device. Consequently, the arcs of $L$ can cross through the ties. We will denote $\T$ the set of oriented tied links in ST.
\end{definition}
\begin{remark}\label{restriction}
  If $I$ is empty, then $L(I)$ is nothing else that a classical link in $ST$. In the same fashion, if $L$ is an affine link in $ST$, and $I$ only contains pairs of points that belong to the standard components, then according to Remark~\ref{affinelinks}, $L(I)$ can be regarded as a tied link in $S^3$. Thus, we have that the set of classical tied links $\mathcal{T}$ from \cite{aiju2} is embedded in $\mathcal{T}_{ST}$.
\end{remark}

Note that the set $I$ induces a partition on the set of the components of $L$, where two components of $L$ belong to the same class if they are connected by a tie.

\begin{definition}
  Let $L(I)$ a tied link in $ST$. A diagram of $L(I)$ is a corresponding mixed link diagram of $L$ in $S^3$ provided with ties connecting pairs of points in the set of ties $I$.
\end{definition}

\begin{definition}\label{defisot}\rm
  Let $L(I), L'(I')$ be two oriented tied links. We say that $L(I)$ and $L'(I')$ are tie isotopic if:
  \begin{itemize}
    \item[(i)] $L$ and $L'$ are isotopic in $ST$ (Section~\ref{knotsB}).
    \item[(ii)] $I$ and $I'$ define the same partition in the set of components of $L$ and $L'$, respectively.
  \end{itemize}
\end{definition}
It is not difficult to check that tie isotopy is an equivalence relation, which is denoted symply by $\sim_t$.\smallbreak
From now on, without risk of confusion, when we say tied link we will refer tied link in $ST$. Additionally, we just write $L$ instead $L(I)$. \smallbreak

Note that tie isotopy says that we can move any tie between two components letting its extremes move along the whole component. Additionally, we can add or remove ties as long as these do not modify the induced partition on the set of components. For instance, we can add or remove:
\begin{itemize}
  \item ties connecting two points of the same component,
  \item ties between components that are already in the same class.
\end{itemize}

Let $L$ be a tied link, and let $c_i, c_j,c_k$ be three different components of $L$. Set points $p_s,p_s'\in c_s$ for $s\in \{i,j,k\}$. The tie isotopy also stand that if we have two ties $(p_i,p_j),(p_j',p_k)\in I$, we then can change indistinctly these ties for $(p_i,p_j)(p_i',p_k)$ or $(p_k,p_i)(p_k',p_j)$. For instance, Figure~\ref{ExampleTL} shows two tie isotopic links. It is clear that the components are ambient isotopic. On the other hand we also have that the corresponding set of ties induces the same partition into their respective components.

\begin{figure}[h!]
\begin{center}
  \includegraphics{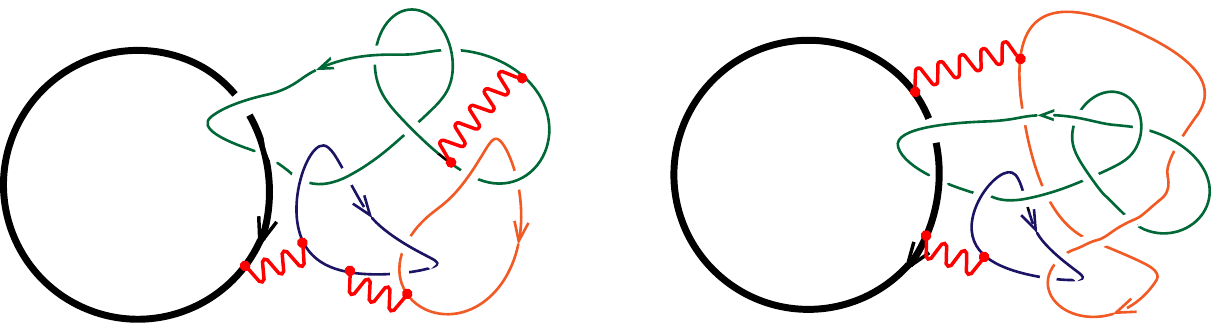}
	\caption{Two links tied--isotopic in ST.}\label{ExampleTL}
	\end{center}
 \end{figure}

\begin{definition}\rm
  We say that a tie is \emph{essential} if this cannot be removed, i.e. removing this tie we obtain a different partition in the set of components.
\end{definition}

For instance, in Figure~\ref{ExampleTL}, the tied link on the left has two essentials ties and one that is not (the tie connecting points in the green component).

\section{An invariant for tied links in ST}\label{sec4}
In this section, we construct an invariant for tied links in $ST$. In order to do that, we need to set notation. From now on, let $\U, \V, \x ,\y, \w, \z$ be indeterminates, and set $\mathbb{K}:=\mathbb{C}(\U, \V, \x ,\y, \w, \z)$. An invariant of ties links is nothing else that a function $\mathcal{F}_{\B}:\T\rightarrow \mathbb{K}$  that is constant in the classes of tie--isotopic links. We define this invariant via skein relations.\\

The following theorem is obtained by readjusting the arguments in \cite{aiju2}.

\begin{theorem}\label{skeinth}
  There exists an invariant of oriented tied links $\mathcal{F}_{\B}:\T\rightarrow \mathbb{K}$ that is uniquely defined by the following conditions:\\
  Let $\bigcirc, \tilde{\bigcirc}, \bigcirc_+$ $\tilde{\bigcirc}_+$ be the tied unknots in the Figure~\ref{unknots}.
  \begin{itemize}
    \item[(i)] \emph{Initial conditions:}  $\f_{\B}(\bigcirc)=1$; $\f_{\B}(\tilde{\bigcirc})=\x$; $\f_{\B}(\bigcirc_+)=\y$; $\f_{\B}(\tilde{\bigcirc}_+)=\w$.
    \item[(ii)] Let $L$ be a tied link. Then we have
    $$\f_{\B}(L\sqcup \bigcirc)= \frac{1}{\z\lambda}\f_{\B}(L),\quad \f_{\B}(L\sqcup \tilde{ \bigcirc})=\frac{\x}{\z\lambda}\f_{\B}(L),\quad \f_{\B}(L\sqcup \bigcirc_+)=\frac{\y}{\z\lambda}\f_{\B}(L)$$

    $$ \f_{\B}(L\sqcup \tilde{\bigcirc}_+)=\f_{\B}(L\tilde{\sqcup} \bigcirc_+)=\frac{\w}{\z\lambda}\f_{\B}(L).$$
  where $\tilde{\sqcup}$ means that we add the corresponding unknot tied together to some standard component of $L$. Additionally, we have:

  $$\f_{\B}(L\tilde{\sqcup} \tilde{\bigcirc}_+)=\frac{\w}{\z\lambda}\f_{\B}(\tilde{L}), \quad \f_{\B}(L\tilde{\sqcup} \tilde{ \bigcirc})=\frac{\x}{\z\lambda}\f_{\B}(\tilde{L})  $$
  where $\tilde{L}$ is the tied link obtained from $L$ by adding a tie from the component that is connected with the unknot added to the fixed component.
    \item[(iii)] \emph{Skein rule I:} Let $L_+, L_-, L_{\sim}, L_{+,\sim}, L_{-,\sim} $ be the diagrams of tied links, that are identical outside the small disk, whereas inside the disk the diagram looks as shown in Figure~\ref{skein1}. Then the following identity holds:
        $$\frac{1}{\lambda}\f_{\B}(L_+)-\lambda\f_{\B}(L_-)=(\U-\U^{-1})\f_{\B}(L_{\sim})$$
        where $\lambda:=\sqrt{\frac{\z - ({\U}- {\U}^{-1})\x}{\z}}$.
    \item[(iv)] \emph{Skein rule II:} Let $L_+^M, L_-^M, L_{\sim}^M, L_{+,\sim}^M, L_{-,\sim}^M $ be the diagrams of tied links, that are identical outside the small disk, whereas inside the disk the diagram looks as shown in Figure~\ref{skein2}. Then the following identity holds:
        $$\f_{\B}(L_+^M)-\f_{\B}(L_-^M)=(\V-\V^{-1})\f_{\B}(L_{\sim}^M)$$
     \end{itemize}
\end{theorem}

\begin{figure}[h!]
\begin{center}
  \includegraphics{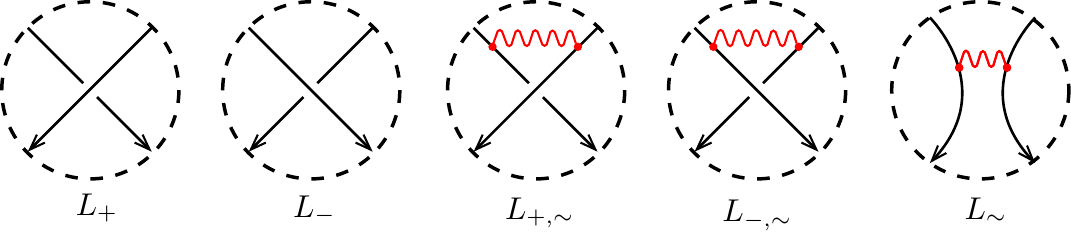}
	\caption{The disks where $L_+, L_-, L_{+,\sim}, L_{-,\sim}, L_{\sim}$ differ.}\label{skein1}
	\end{center}
 \end{figure}

\begin{figure}[h!]
\begin{center}
  \includegraphics{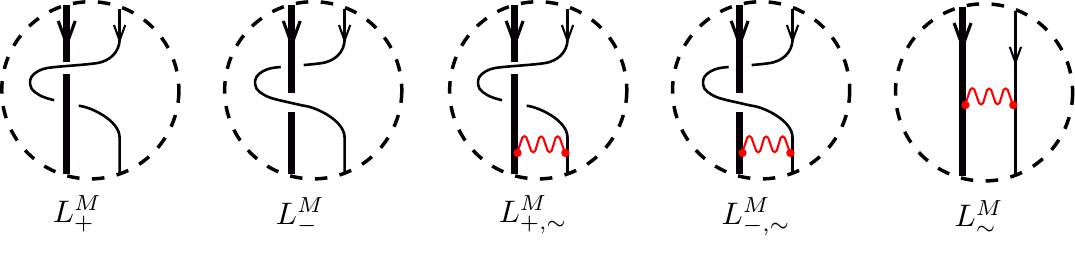}
	\caption{The disks where $L_+^M, L_-^M, L_{+,\sim}^M, L_{-,\sim}^M, L_{\sim}^M$ differ.}\label{skein2}
	\end{center}
 \end{figure}

 \begin{figure}[h!]
\begin{center}
  \includegraphics{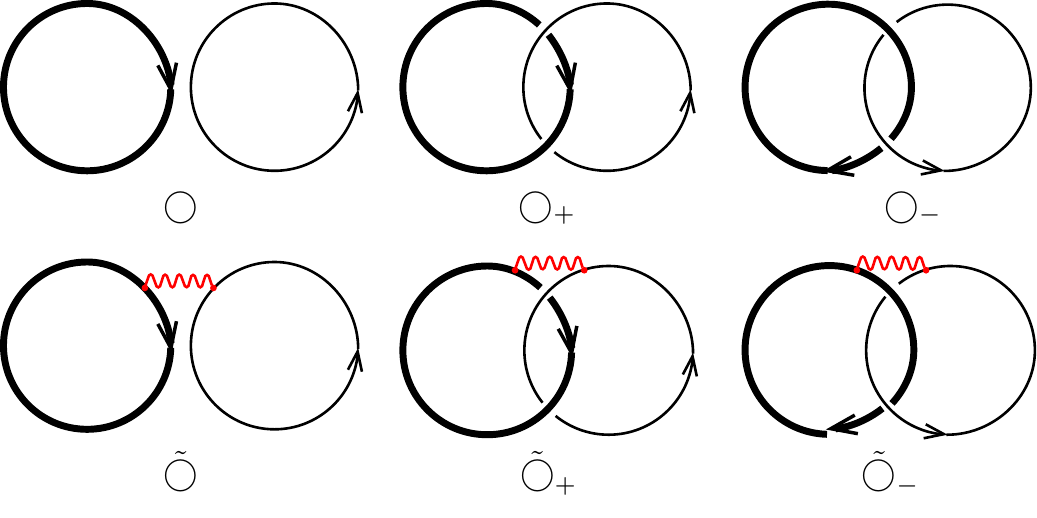}
	\caption{The six tied--unknots in ST.}\label{unknots}
	\end{center}
 \end{figure}

\begin{remark}
  Skein rules (ii) and (iii) imply the following skein rules:
   \begin{itemize}
     \item[(v)]  $\frac{1}{\lambda}\f_{\B}(L_{+,\sim})-\lambda\f_{\B}(L_{-,\sim})=(\U-\U^{-1})\f_{\B}(L_{\sim}),$
     \item[(vi)]$\f_{\B}(L_{+,\sim}^M)-\f_{\B}(L_{-,\sim}^M)=(\V-\V^{-1})\f_{\B}(L_{\sim}^M),$
   \end{itemize}
 which are obtained by adding a tie between the two strands inside the disc in each case.
\end{remark}

\begin{proof}
We proceed by following the proof of \cite{limi}. The proof has some slight changes when ties and loops around the fixed component are involved.\\

Let $\T^n$ be the set of diagram of $n$ crossings (recall Season~\ref{knotsB}), and let $L$ be in $\T^n$. It is well known that we can associate to $L$ an ascending diagram $L'$. To obtain this diagram, we first have to order the components and fix a base point on each of them. Then $L'$ is obtain by starting at the base point of the first component and changing all the overpasses to underpasses along the component. We then do the same process for the subsequent components. Thus, we obtain a diagram that every crossing is first encountered as an underpass. This process separates and unknots the components. Eventually the components of $L'$ have loops around the fixed component. Without loss of generality, we can assume that all are positive or negative, since two consecutive loops with opposite sign are isotopic to a segment that does not have loops around the fixed component (see Figure~\ref{ReidM}). Then, we define the positive ascending diagram $L'_+$ as the diagram that is obtain from $L'$ by changing the loops of the components of $L'$. We proceed as follows:\\
\begin{itemize}
  \item If a component of $L'$ has only positive loops. We leave the first loop (according to the orientation of the component) unaltered, the second one is change by a negative loop. We then do the same with the fourth and so on.
  \item If the component has only negative loops we proceed analogously.
\end{itemize}

Thus, if a component has $2n$ loops in the diagram $L'$, this will have $n$ couples of consecutive loops with opposite sign in the diagram $L'_+$. Analogously, if the number of loops is $2n+1$, the corresponding component $L'_+$ will have $n$ couples of consecutive loops with opposite sign and a positive or a negative loop at the end.
   We thus have that $L'_+$ is a disjoint union of tied unknots $\bigcirc, \bigcirc_-, \bigcirc_+$, which are tied together according to the initial ties in the link $L$. It is clear that $L$ and $L'_+$ just differ in a finite number of crossings and loops, called \lq\lq deciding crossings\rq\rq \ (deciding loops, respectively), where the signs in those crossing and loops are opposites. % Additionally, we define $L'_{\sim}$ as the link that is obtained from $L'$ by adding a tie in every deciding crossing.
This procedure allows to get an ordered sequence of deciding points, whose order depends from the ordering of the components, and the choice of base points.\\
We now proceed by induction in the number of standard crossings. We thus assume that the function $\f_{\B}:\T^n\rightarrow \mathbb{K}$ satisfies the relations (i)-(iv), is independent of the ordering of the points, and of the choices of base points as well. Also, $\f_{\B}$ is invariant under Reidemeister moves. %and mixed moves do not increase the number of standard crossings beyond $n$.
 Moreover, for any disjoint union of tied unknots on Figure~\ref{unknots}, the value of $\f_{\B}$ may be computed by using rules (i) and (ii). \\

We start with zero crossings. Thus, the tied link is a disjoint union of tied unknots. And we know the value of $\f_{\B}$ in this case.\smallbreak

Let $L$ be in $\T^{n+1}$. If $L$ is a disjoint union of tied--unknots the result follows. Otherwise, consider the first deciding crossing $p$. If in a neighborhood of $p$ the tied link looks like $L_+$ (or $L_-$), we can use the skein rule (iii) for writing the value of $\f_{\B}$ in terms of $L_-$ (or $L_+$) and $L_{\sim}$. %Analogously, if in a neighborhood of $P$ the tied link looks like $L_{+,\sim}$ (or $L_{-,\sim}$), we can use skein rules $(iiiA)$ for deduce the value of $\f_{\B}$ in terms of $L_{-,\sim}$ (resp. $L_{+,\sim}$) and $L_{\sim}$.
 Then we apply the same procedure on the second deciding crossing and so on. Finishing this process, we proceed to do the same with the deciding loops though using skein rule (iv) (or (vi)). Remember that if the a loop looks like $L_+^M$ (or $L_-^M$), we can use the skein rule (iv) for writing the value of $\f_{\B}$ in terms of $L_-^M$ (or $L_+^M$) and $L_{\sim}^M$. Analogously, if the loop looks like $L_{+,\sim}^M$ (or $L_{-,\sim}^M$), we can use skein rules (vi) for deducing the value of $\f_{\B}$ in terms of $L_{-,\sim}^M$ (resp. $L_{+,\sim}^M$) and $L_{\sim}^M$.

Thus, at the end of the process, we have express $\f_{\B}(L)$ in terms of $\f_{\B}(L'_+)$ and two other tied links that are a disjoint union of unknots tied together in some way. For these unions the value of $\f_{\B}$ is known and only depends of the number of components and the number of essential ties. Thus, it remains to prove that:
\begin{itemize}
  \item[(i)] the procedure is independent of the order of the deciding crossings and deciding loops.
  \item[(ii)] the procedure is independent of the order of the components, and from the choice of base points.
  \item[(iii)] the function is invariant under Reidemeister moves.
\end{itemize}

The skein rule (iii) is similar to the skein rule used in \cite{limi} (Homflypt type). Indeed, just the link of right part of the equality changes, including a tie between the strands. We then omit the proofs of (i)--(iii), since these follow almost directly by slightly modifying the corresponding proofs given in \cite{limi}.
\end{proof}

%\begin{remark}\rm
%  Note that the skein (iv) and (vi) are only used for computing the value of $\f_{\B}$ in links that has negative loops. More precisely, let $L$ be tied link with negative loops. Then, we have
%  $$\f_{\B}(L)=p\f_{\B}(L')+q\f_{\B}(L_0),$$
%  where $L'$ is the corresponding ascending diagram, $L_0$ is a disjoint union of unknots of Figure~\ref{unknots}, and $p$ and $q$ are polynomials in the variables $\U, \V, \x ,\y, \w, \z, \lambda$. It is clear that $L'$ and $L_0$ have the same loops of $L$, consequently they have unknots of the form  $\bigcirc_4$ (or $\bigcirc_5$). Using (iv) (resp. (vi)), we can obtain the value of these unknots in terms of the unknots $\bigcirc_1$ and $\bigcirc_2$ (resp. $\bigcirc_3$), where the values are known by (i).
%\end{remark}
\begin{remark}\label{restriction2}
  Let $\mathcal{F}$ be the invariant defined for tied links in $S^3$ in \cite[Section 2]{aiju2}. By Remark~\ref{restriction}, we have that $\mathcal{F}_B$ restricted to affine tied links in $\mathcal{T}_{ST}$ is equivalent to invariant $\mathcal{F}$.
\end{remark}

Recall from \cite{aiju2} that $\f$ holds the following properties:
\begin{itemize}
  \item[(i)] $\f$ is multiplicative with respect to the connected sum of tied links.
  \item[(ii)]  The value of $\f$ does not change if the orientations of all curves of the link are
reversed.
  \item[(iii)] Let $L$ be a link diagram whose components are all tied together, and $L^+$
be the link diagram obtained from $L$ by changing the signs of all crossings. Thus, $\f(L^+)$ is obtained from $\f(L)$ by the following changes: $ \lambda \rightarrow  1/\lambda$ and $\U \rightarrow 1/\U$.
\end{itemize}

It is not difficult to check that the invariant $\mathcal{F}_{\B}$ just satisfies an analogue of property (iii) above. More precisely, let $L$ be a tied link whose standard components are all tied together, and $L^*$ be the link diagram obtained from $L$ by changing the signs of all crossings. Thus $\f_{\B}(L^*)$ is obtained from $\f_{\B}(L)$ by doing the change: $\lambda \rightarrow 1/\lambda$, $\U \rightarrow 1/\U$., $\V \rightarrow 1/\V$.\smallbreak

On the other hand, unlike the classical case, there is no well-defined operation of connected sum for knots in $ST$ (see \cite{grav}). Thus, we do not have an analogous for (ii). Additionally, observe that the value of $\f_{\B}$ is not invariant if we reverse the orientation of all the components of the links. Indeed, we have that $\f_{\B}(\bigcirc_2)\not=\f_{\B}(\bigcirc_4)$. However, if we consider a tied link $L$ without loops, then $\f_{\B}(L)$ does not change if we reverse the orientation (cf. \cite[Section 2.2]{aiju2}).

\begin{example}\rm
  Let $H^+, \dot{H}^+,\tilde{\dot{H}}^+, H^+_{\ell}, \dot{H}^+_{\ell},\tilde{\dot{H}}^+_{\ell}$ be the tied links in the Figure~\ref{example}.

   \begin{figure}[h!]
\begin{center}
  \includegraphics{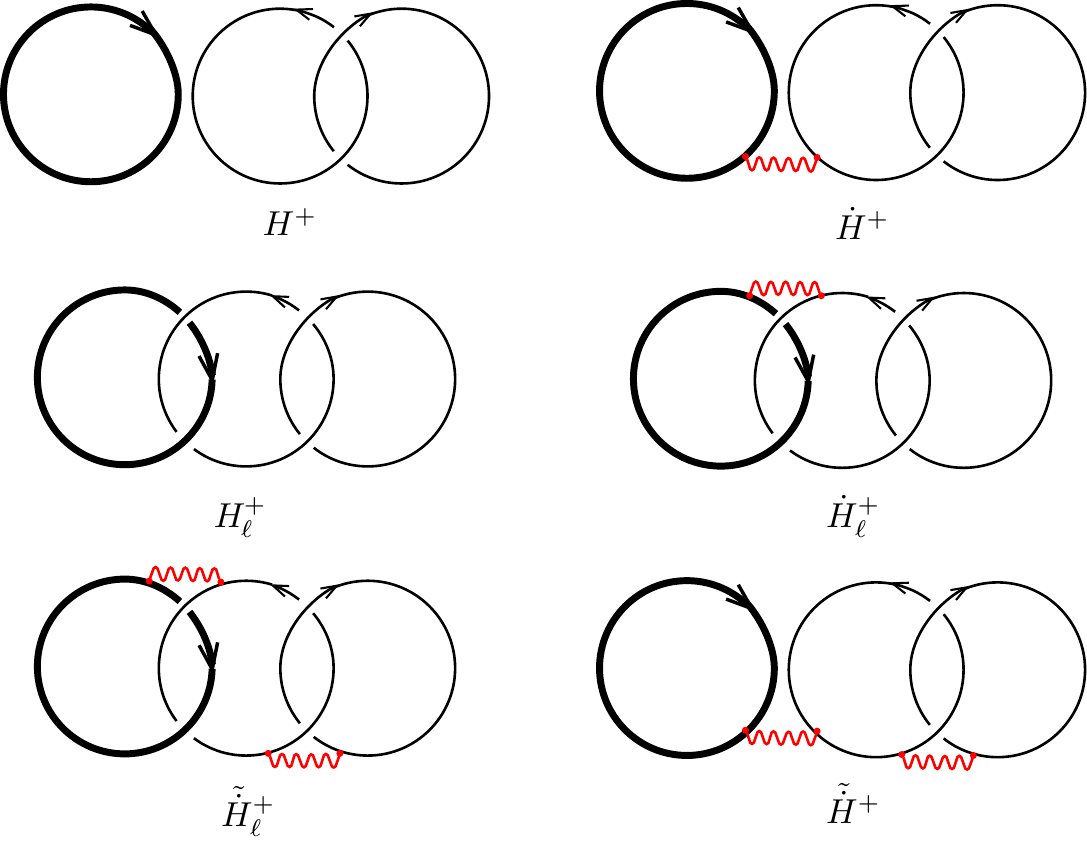}
	\caption{Different Hopf links in ST}\label{example}
	\end{center}
 \end{figure}

  We next compute the value of $\f_{\B}(\tilde{\dot{H}}^+_{\ell})$. Using (ii), we obtain
  $$\f_{\B}(\tilde{\dot{H}}^+_{\ell})=\lambda^2 \f_{\B}(H_1)+\lambda(\U-\U^{-1})\f_{\B}(H_2),$$
  where $H_1$ and $H_2$ are the tied links in Figure~\ref{skeinex}.

   \begin{figure}[h!]
\begin{center}
  \includegraphics{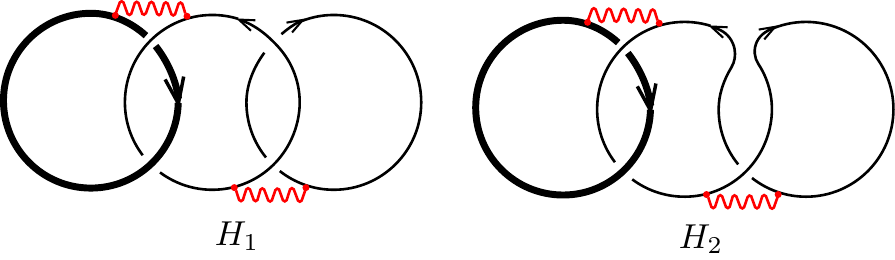}
  \caption{}\label{skeinex}
		\end{center}
 \end{figure}

  Therefore, we have:
  \begin{eqnarray*}
  % \nonumber % Remove numbering (before each equation)
    \f_{\B}(\tilde{\dot{H}}^+_{\ell}) &=& \lambda^2\frac{\w\x}{\z\lambda}+\lambda(\U-\U^{-1})\w \\
     &=& \lambda\w(\frac{\x}{\z}+\U-\U^{-1})=\frac{\lambda\w(\x\U+\U^2\z-\z)}{\U\z}
  \end{eqnarray*}

  We can compute the polynomial of the others tied links in Figure~\ref{example} by an analogous way. More precisely, we have:
  \begin{eqnarray*}
  % \nonumber % Remove numbering (before each equation)
   \f_{\B}(H^+) &=& \frac{\lambda(\U+\U^2\z-\z)}{\z\U},\quad \f_{\B}(\dot{H}^+)  = \frac{\lambda\x(\U+\U^2\z-\z)}{\z\U}  \\
    \f_{\B}(\tilde{\dot{H}}^+)&=& \frac{\lambda\x(\x\U+\U^2\z-\z)}{\z\U},\quad \f_{\B}(H^+_{\ell}) =  \frac{\lambda\y(\U+\U^2\z-\z)}{\z\U} \\
     \f_{\B}(\dot{H}^+_{\ell})&=& \frac{\lambda\w(\U+\U^2\z-\z)}{\U\z}
  \end{eqnarray*}
  \end{example}

%  \begin{remark}\rm
 % Note that if we write $\lambda$ in terms of $\z, \U$ and $\x$ as follows:
  % $$\lambda := \sqrt{\frac{\z - ({\U}- {\U}^{-1})\x}{\z}},$$
  % we can express the polynomial $\f$ as a Laurent polynomial in variables $\U,\V,\z,\w,\y,\x$ multiplied by $\lambda^{\epsilon}$, with $\epsilon=-1, 0, 1$. This will be useful in the last section, where we recover $\f$ using the Jones' method (cf. \cite[Remark 2.5]{aiju2}).  \end{remark}

\section{The tied braid monoid of type $\B$}\label{sec5}

In this section, we introduce the tied braid monoid of type $\B$ in order to obtain analogues for Alexander and Markov theorems for tied links in $ST$. This, with the aim of recovering $\f_{\B}$ via Jones' method using the algebra of braids and ties of type $\B$ and the respective Markov trace defined in \cite{fl} .\smallbreak

We begin introducing the tied braid monoid of type $\B$ and giving the corresponding diagrammatical interpretation.

\begin{definition}\rm
  We define the tied braid monoid of type $\B$, denoted by $TB_n^{\B}$, as the monoid generated by $\rho_1,\sigma_1,\dots , \sigma_{n-1}$, the usual braid generators of $\B$--type, and the generators $\phi_1, \eta_1,\dots ,\eta_{n-1}$, called ties, satisfying the relations (\ref{braidB}) of $\widetilde{W}_n$ together with the following relations:

\begin{eqnarray}
    \eta_i\eta_j&=&\eta_j\eta_i  \quad \mbox{for all $i, j$},\label{typeA1}\\
    \eta_i\sigma_i&=&\sigma_i\eta_i   \quad \mbox{for all $1\leq i\leq n-1$},\label{typeA2}\\
    \eta_i\sigma_j&=&\sigma_j\eta_i \quad \mbox{for all $|i-j|>1$},\label{typeA3} \\
    \eta_i\sigma_j\sigma_i&=&\sigma_j\sigma_i\eta_j\quad \mbox{for all $|i-j|=1$},\label{typeA4}\\
    \eta_i\sigma_j\sigma_i^{-1}&=&\sigma_j\sigma_i^{-1}\eta_j\quad \mbox{for all $|i-j|=1$},\label{typeA5}\\
    \eta_i\eta_j\sigma_i&=&\sigma_i\eta_i\eta_j= \eta_j \sigma_i \eta_j \quad \mbox{for all $|i-j|=1$},\label{typeA6}\\
    \eta_i^2 &=& \eta_i  \quad \mbox{for all $i$}, \label{typeA7}\\
    \phi_1^2&=&\phi_1 \quad  \label{typeB1} \\
    \rho_1\eta_i&=&\eta_i\rho_1\quad \mbox{for all $i$}, \label{typeB2}\\
    \rho_1\phi_1&=&\phi_1\rho_1 \label{typeB3}\\
    \phi_1\eta_i&=&\eta_i\phi_1 \quad \mbox{for all $i$},\label{typeB4}\\
    \phi_1\sigma_i&=&\sigma_i\phi_{1} \label{perm}\quad \mbox{for all $2\leq i\leq n-1 $},\label{typeB5}\\
    \sigma_{i-1}\dots\sigma_1\phi_1\sigma_1^{-1}\dots\sigma_{i-1}^{-1}&=& \sigma_{i-1}^{-1}\dots\sigma_1^{-1}\phi_1\sigma_1\dots\sigma_i\label{ef}\quad \mbox{for all $2\leq i\leq n$}.\label{typeB6}\\
    \phi_1\eta_1&=&\phi_1\sigma_1\phi_1\sigma_1^{-1}=\sigma_1\phi_1\sigma_1^{-1}\eta_1 \label{typeB7}
\end{eqnarray}
\end{definition}

\begin{remark}
Note that relations (\ref{typeA1})--(\ref{typeA7}) are exactly the defining relations of $TB_n$, the briad tied monoid defined in \cite{aiju2}. Therefore, we have $TB_n\leq TB_n^B$. %We will refer to those relations as relations of type $\mathtt{A}$.
\end{remark}

\begin{remark}
  For $n\geq 1$, we have that $\TB \subseteq TB_{n+1}^B$. Then, we can define $TB_{\infty}^B$ as the inductive limit $\sqcup_{n\geq 1} \TB$.
\end{remark}

 From now on, the relations of $TB_n$ will be called type--$\mathtt{A}$ relations, and the rest of them type--$\B$ relations.

In terms of diagrams, the generators $\rho_1,\sigma_1,\dots , \sigma_{n-1}$ represent the usual braid generators of type $\mathtt{B}$ (see Figure~\ref{gentypeB}). On the other hand, the defining generator $\phi_1$ corresponds to the braid of type $\mathtt{B}_n$ that has a tie connecting the fixed strand and the first moving strand, whereas $\eta_i$ is represented by the $\mathtt{B}$--type braid that has a tie connecting the $i$--th and $(i+1)$--st moving strands. See Figure~\ref{generatorsfig} for this identification.\smallbreak

\begin{figure}
  \centering
  \includegraphics{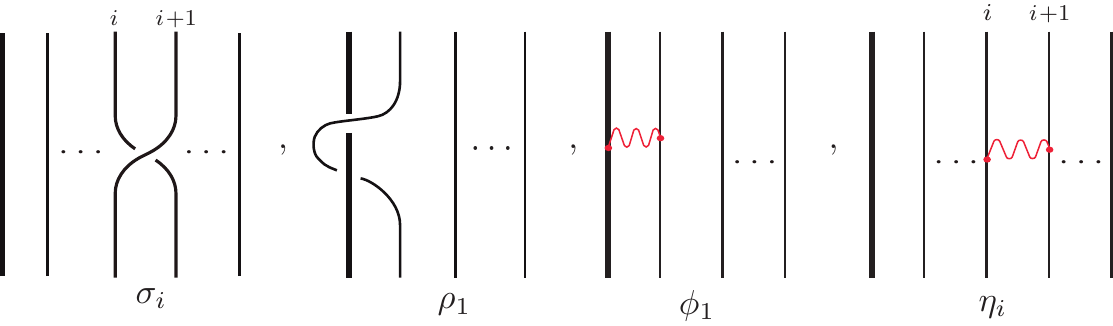}
  \caption{Diagrams corresponding to the generators of $\TB$}\label{generatorsfig}
\end{figure}

\begin{figure}
\centering
 \includegraphics{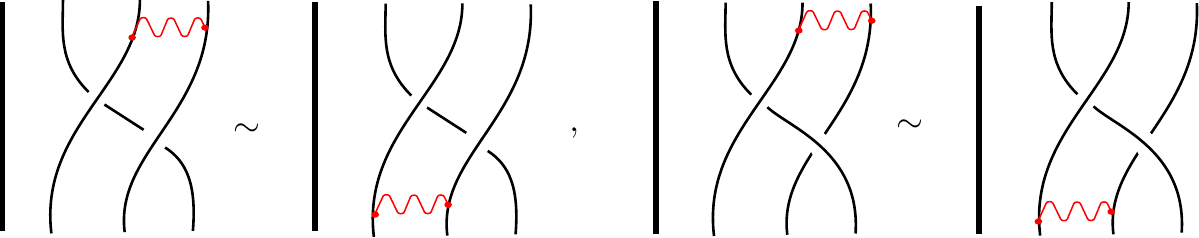}
 \caption{Relations (\ref{typeA4}) and (\ref{typeA5}) in terms of diagrams ($n=3$)}\label{tiedA}
\end{figure}

 The defining relations may also be expressed in terms of diagram. For instance, recall from \cite{aiju2} that relation (\ref{typeA4}) corresponds to move the tie from top to bottom behind or in front of the strand (see Figure\ref{tiedA}). For more details about type-$\mathtt{A}$ relations see \cite[Section 3.1]{aiju2}).
\subsection{Generalized ties}
 Let $\eta_{i,j}, \phi_{j}$ be the elements $\TB$ defined as follows:

 \begin{equation}\label{genties}
\begin{array}{rcll}
\eta_{i,j} & =  & \sigma_i\cdots\sigma_{j-2}\eta_{j-1}\sigma_{j-2}^{-1}\cdots \sigma_{i}^{-1} & \text{ for} \quad \vert i-j\vert >1,\\
\phi_{j} & = & \sigma_{j-1}\cdots\sigma_{1}\phi_{1}\sigma_{1}^{-1}\cdots \sigma_{j-1}^{-1} &\text{ for} \quad 2 \leq j\leq 1.
  \end{array}
\end{equation}
 where, by convention $\eta_{i,i}=1$ and $\eta_{i,i+1}=\eta_i$
We begin recalling some known facts about the elements $\eta_{i,j}$'s from \cite{aiju2}. By definition, we have that $n_{1,3}$ corresponds to the diagram in the top left of Figure~\ref{gentiesa}.\smallbreak

\begin{figure}
  \centering
  \includegraphics{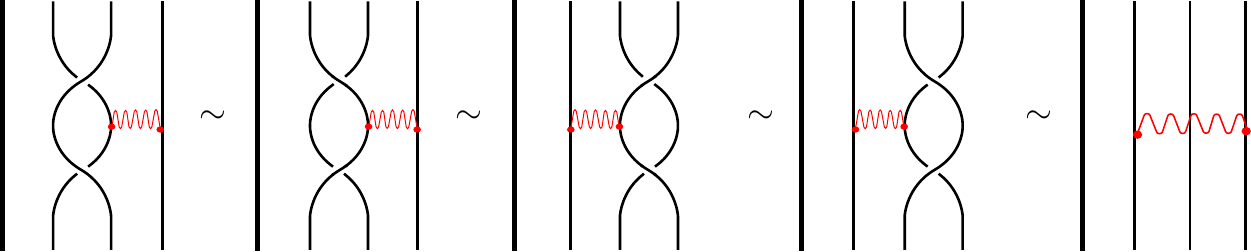}
  \caption{Equivalent diagrams of $\eta_{1,3}$ according Eq. (\ref{relaciongenties})}\label{gentiesa}
\end{figure}

Then, observed that, if the tie is provided with elasticity, we may transform such diagram into the diagram in top right of Figure~\ref{gentiesa} by using a Reidermeister move of second type. Thus, we consider ties as elastic objects, and therefore, they are represented as a spring.\smallbreak
More generally, using the defining relations of $TB_n^{\B}$, we know that there are $2^{k-i-1}$ equivalent expression for $n_{i,k}$. Specifically, given a pair $i,k$, such that $k-1>1$, we have that:
\begin{equation}\label{relaciongenties}
  n_{i,k}=s_is_{i+1}\cdots s_{k-2}\eta_{k-1}s_{k-2}^{-1}\cdots s_{i+1}^{-1}s_i^{-1}
\end{equation}
\noindent for all possible choices of  $s_l=\sigma_l$ or $s_l=\sigma_l^{-1}$ (see \cite[Section 3.2]{aiju2} for details). Thus, we have that the elements $\eta_{i,j}$ diagrammatically corresponds to an elastic tie joining the $i$-th moving strand with the $j$-th moving strand. \smallbreak
%It is also proved in \cite{aiju2} that the elements $n_{i,j}$'s satisfy the following relations:

%which implies that ties can be moved at the end or the beginning of an element $\alpha\in \TB$ (cf. \cite[Proposition 3.4]{aiju2})

Additionally, we have that the following relations hold:
\begin{equation}\label{movepro}
 \begin{array}{cc}
   \sigma_i \eta_{i,j}= \eta_{i+1,j}\sigma_i & \sigma_j \eta_{i,j}= \eta_{i,j+1}\sigma_i  \\
  \sigma_{i-1} \eta_{i,j}= \eta_{i-1,j}\sigma_{i-1} &\sigma_{j-1} \eta_{i,j}= \eta_{i,j-1}\sigma_{j-1}
\end{array}
\end{equation}

and
\begin{equation}\label{equivA}
n_{i,k}n_{k,m}=n_{i,k}n_{i,m}=n_{k,m}n_{i,m}\quad \mbox{for all $1 \leq i,k,m\leq n$  }
\end{equation}
%which induces a partition in the set of n moving strands (cf. \cite[Proposition 3.3]{aiju2}).

On the other hand, we can obtain similar results for the ties that are connected to the fixed strand. Indeed, for $i =2$ the relation (\ref{typeB6}) corresponds to the diagram in Figure~\ref{gentiesb}. Then, by using a Reidermeister move of second type, we also may consider that the tie is elastic (as in the type $\mathtt{A}$ case).  Thus, using induction, we have that the element $\phi_j$ diagrammatically corresponds to a tie joining the fixed strand and the $j$-th moving strand. Additionally, the elements $\phi_j's$ satisfy the following relations:

\begin{eqnarray}
% \nonumber % Remove numbering (before each equation)
\phi_j\eta_i&=&\eta_i\phi_{j} \quad \mbox{for all $1\leq i \leq n-1$ and $1\leq j\leq n$} \label{typeB1G}\\
\phi_j\sigma_i&=&\sigma_i\phi_{s_i(j)}\quad \mbox{where $s_i$ is the transposition $(i\ i+1)$} \label{typeB2G} \\
  \eta_{i,j}\phi_i &=& \phi_i\phi_j =\phi_j\eta_{i,j} \quad \mbox{for all $1\leq i\not=j \leq n$ }   \label{equivB}
\end{eqnarray}
Indeed, (\ref{typeB1G}) and (\ref{typeB2G}) follow directly by using defining relations (\ref{typeB4})--(\ref{typeB6}). And, we obtain (\ref{equivB}) by conjugating the defining relation (\ref{typeB7}) by the element $(\sigma_{i-1}\dots\sigma_{1})(\sigma_{j-1}\dots\sigma_1)$, whenever $i<j$, which we can suppose without loss of generality. Then, these relations correspond to the diagrams in Figure~\ref{gentiesb2}.

\begin{figure}
  \centering
  \includegraphics{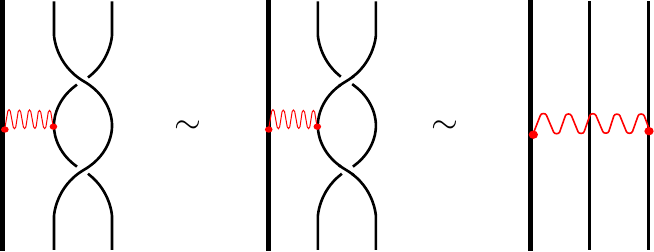}
  \caption{Relation (\ref{typeB6}) in terms of diagrams}\label{gentiesb}
\end{figure}

\begin{figure}
  \centering
  \includegraphics{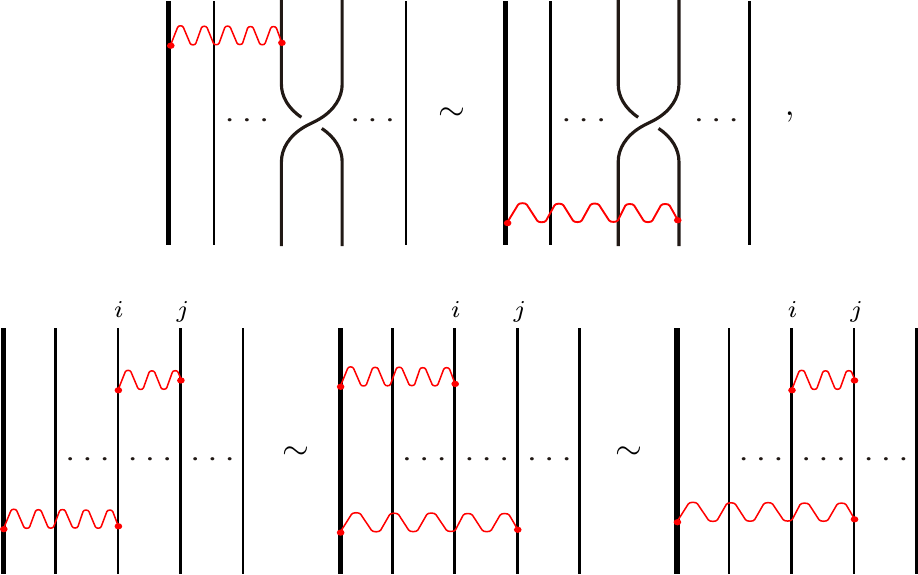}
  \caption{Relations (\ref{typeB2G}) and (\ref{equivB}) in terms of diagrams}\label{gentiesb2}
\end{figure}

Let $TB_n^{\sim}$ be the submonoid of $TB_n^{\B}$ generated by the elements $\eta_{i,j}, \phi_{j}$ with $1\leq i,j \leq n$.  Thus, using the preceding results, we have the following proposition

\begin{proposition}\label{movility}
  Let $\alpha$ be a tied briad in $\TB$. Then, $\alpha$ can be written by $\alpha=\gamma\beta$ (or $\alpha=\beta\gamma'$), where $\beta$ is a braid of type $\mathtt{B}$, and $\gamma$ ( or $\gamma$') is in $TB_n^{\sim}$. (cf. \cite[Proposition 3.2]{aiju2})
\end{proposition}
\begin{proof}
 The result follows easily by using relations (\ref{movepro}) and (\ref{typeB2G}).
\end{proof}

\begin{proposition}\label{partition}
 Let $\gamma$ be an element of $TB_n^{\sim}$. Then, $\gamma$ defines a equivalence relation in the set of n+1 strands (including the fixed strand). (cf. \cite[Proposition 3.3]{aiju2})
\end{proposition}
\begin{proof}
  The properties of reflexivity and symmetry are direct, whereas the transitivity property is implied by relations (\ref{equivA}) and (\ref{equivB})
\end{proof}

\begin{remark}\rm
  Let $\alpha=\gamma\beta$ be an element of $\TB$, where $\beta$ is a braid and $\gamma\in TB_n^{\sim}$. By Proposition~\ref{partition}, $\gamma$ induces a partition in the set of strands of $\beta$, or equivalently, a set-partition of $\{0,1,\ldots,n\}$, where $0$ represents the fixed strand, and $i$ represents the $i$-th moving strand, for $1\leq i\leq n$. (cf. \cite[Proposition 3]{fl})
\end{remark}

\section{The Alexander and Markov theorems for tied links in ST}\label{sec6}
The closure of a tied braid $\alpha$ in $\TB$, denoted by $\widehat{\alpha}$, is defined analogously as closure in $\widetilde{W}_n$ (see Section~\ref{knotsB}). Clearly, the result of closure $\widehat{\alpha}$, is a tied link in $ST$. Thus, we have a map $\widehat\ :TB_{\infty}^{\mathtt{B}} \rightarrow \T$. In this section, we prove that this map is surjective. We then define a set of Markov moves in $TB_{\infty}^{\mathtt{B}}$ in order to prove a Markov theorem for tied links.
\begin{theorem}\label{Alexander}\textbf{ (Alexander theorem for tied links in ST)}
Let $L$ be a link in $\T$. Then, there is $\alpha \in \TB$ such that $L=\widehat{\alpha}$.
\end{theorem}
\begin{proof}
Let $L$ be a tied link in $\T$. Recall from Section~\ref{knotsB} that $L$ can be regarded as a mixed link with ties. Then, we apply the algorithm proposed by S. Lambropoulou (see \cite[Section~2.1]{la1}) ignoring the ties. To do that, roughly speaking, we fix $O$ as the center of the fixed component. Then, we apply the Alexander procedure, thought maintaining the fixed component unaltered. Eventually, the resulting link could have ties connecting points in opposite sides from $O$. However, using that the ties ends can move freely along the strands and the transparency property, we can arrange them such that they lie in an annulus centered in $O$ (see Figure~\ref{Alexanderfig}). Finally, we obtain a tied braid by cutting along a half line with origin $O$. This tied braid is by construction tie isotopic to $L$.
\end{proof}

\begin{figure}
  \centering
  \includegraphics{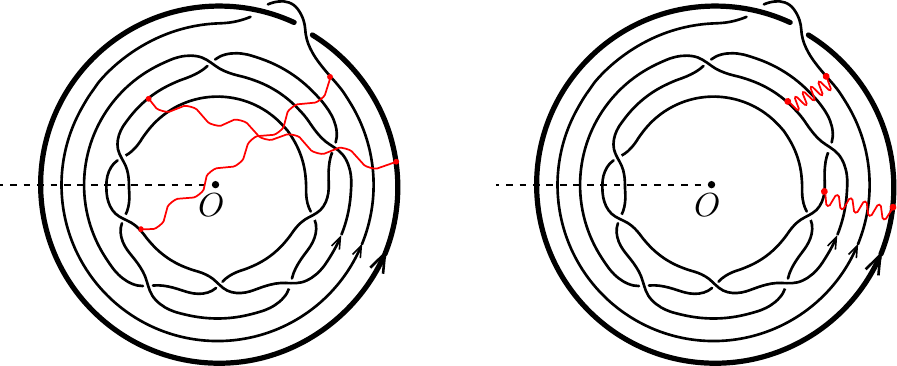}
  \caption{Rearrange of ties in Alexander Theorem }\label{Alexanderfig}
\end{figure}

In the following $\tau_\alpha$ denotes the image of $\alpha$ through the natural homomorphism from $W_n$ into the symmetric group $S_n$ (see Section~\ref{braidsB}).

\begin{definition}
Two tied braids in $TB_{\infty}^{\mathtt{B}}$ are $\sim_{M}$-equivalent if one can be obtained from the other by applying a finite sequence of the following moves:
\begin{itemize}
  \item[(i)] $\alpha\beta$ can be exchanged by $\beta\alpha$
  \item[(ii)] $\alpha\sigma$ can be exchanged by $\alpha\sigma_n$ or $\alpha\sigma_n^{-1}$
  \item[(iii)] $\alpha$ can be exchanged by $\eta_{i,j}\alpha$, if $\tau_{\alpha}(i)=j$.
  \item[(iv)]$\alpha$ can be exchanged by $\phi_j\alpha$, whenever $\tau_{\alpha}(i)=j$ and $\alpha$ constains $\phi_i$.
\end{itemize}
\end{definition}

If $\alpha$ and $\beta$ in $TB_{\infty}^{\mathtt{B}}$ are $\sim_{M}$-equivalent, we write $\alpha\sim_{M}\beta$.

\begin{theorem}\label{Markov}
 \textbf{ (Markov theorem for tied links in ST)}
 Let $\alpha_1, \alpha_2$ be tied braids in $\TB$. Then, the links $L_1=\hat{\alpha_1}$ and $L_2=\hat{\alpha_2}$ are tied isotopic if and only if $\alpha_1\sim_{M}\alpha_2$.
 \end{theorem}
 \begin{proof}
  Firstly, note that considering the ties properties (elasticity, transparency), we can proceed for tied links as in the proof of Markov theorem for classical links in $ST$ (see \cite[Theorem 3]{la1}).\smallbreak Let $\alpha_1$ and $\alpha_2$ be tied braids in $\TB$. Thus, we have that $\alpha_1=\gamma_1\beta_1$ and $\alpha_2=\gamma_2\beta_1$ according to Proposition~\ref{movility}.
  %, where $\beta_1,\beta_2$ are braids, and $\gamma_1,\gamma_2$ are generalized ties.
    Set $L_i=\widehat{\alpha_i}$, for $i=1,2$, and suppose that $L_1$ and $L_2$ are isotopic tied links. We have to prove that $\alpha_1\sim_M \alpha_2$. By Definition~\ref{defisot}, we have that $L_1$ and $L_2$ are isotopic as links in $ST$. Thus, we have that $\alpha_1$ and $\alpha_2$ are related by moves of type (i) and (ii), which coincide with the classical Markov moves in ST. Thus, we have that $\beta_1$ and $\beta_2$  are $\sim_M$--equivalent. More precisely, we can transform $\beta_1$ into $\beta_2$ by using (i) and (ii) moves. Thus, after applying this moves, we have that $\beta_1$ and $\beta_2$ consist in the same braid, denoted by $\beta$. \smallbreak

\begin{figure}
  \centering
  \includegraphics{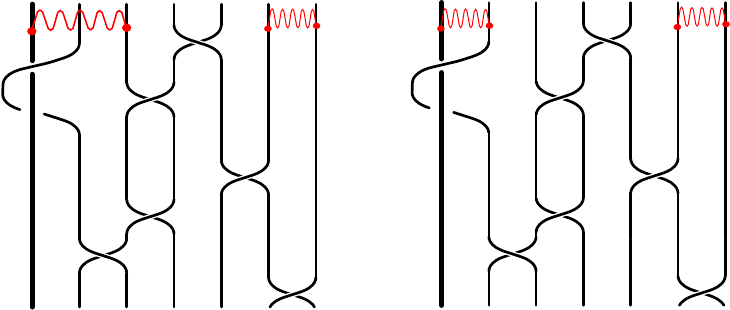}
  \caption{Two different tied braids that have the same closure}\label{ejtied}
\end{figure}

 Therefore, by now, we have that $\alpha_i\sim_M \gamma_i\beta$, for $i=1,2$. Since $L_1\sim_t L_2$, we also know that the set of ties corresponding to $\alpha_1$ and $\alpha_2$ define the same partition in the set of components of $\widehat{\beta}$. However, this fact does not imply that $\gamma_1=\gamma_2$ (for instance, see Figure~\ref{ejtied} ). Therefore, it is enough to prove that we can transform $\gamma_1$ into $\gamma_2$ by applying moves of type (iii) and (iv). If $\gamma_1$ and $\gamma_2$ just contain ties joining the moving strands. By \cite[Theorem 3.7]{aiju2}, we have that $\gamma\gamma_1=\gamma\gamma_2$, where
 \begin{equation}\label{estabA}
   \gamma=\prod_{\tau_{\beta}(i)=j}\eta_{i,j}.
 \end{equation}
That is, $\gamma_1\sim_M\gamma_2$ by using move (iii).
 We now suppose that $\gamma_1$ and $\gamma_2$ have some tie interacting with the fixed strand. Let us say that $\gamma_1$ contains $\phi_i$. Let $c_i$ be the cycle of $\tau_{\beta}$ that contains $i$. Then, $\gamma_2$ must contain $\phi_i$ or $\phi_j$, for some $j$ in the cycle $c_i$, since, $\gamma_1$ and $\gamma_2$ define the same partition in the set of components of $\widehat{\beta}$. For a cycle $c_k$ of $\tau_{\beta}$, we define
 $$\delta_{c_k}:=\prod_{j\in c_k}\phi_j.$$ Now, set

 $$\delta=\prod_{\phi_i\in \gamma_1} \delta_{c_{i}},$$
 where $c_i$ is the cycle of $\tau_{\beta}$ containing $i$. Then, we have that $\delta\gamma\gamma_1=\delta\gamma\gamma_2$, where $\gamma$ is the element from (\ref{estabA}). Thus, $\gamma_1\sim_M\gamma_2$ by using moves (iii) and (iv).
 \end{proof}
\section{The invariant $\mathcal{F}_{\B}$ via Jones' method}\label{sec7}
 The goal of this section is to recover the invariant $\mathcal{F}_{\B}$ by using Jones method. Firstly, observe that, by applying Theorems \ref{Alexander} and \ref{Markov}, we have a correspondence between isotopy classes of $\mathcal{T}_{ST}$ and the set of equivalence classes in $TB^{\mathtt{B}}_\infty$ (according to $\sim_M$). Secondly, we define a natural representation from $TB^{\mathtt{B}}_n$ into the algebra an algebra of braids and ties of type $\mathtt{B}$ \cite{fl}. This algebra supports a Markov trace, hence we may apply Jones' method to obtain the invariant $\overline{\Delta}_{\B}$. We then probe that this invariant is equivalent to $\f_{\B}$ from Section~\ref{sec4}.

\subsection{An algebra of braids and ties of type $\mathtt{B}$} We begin recalling the definition of the algebra introduced in \cite{fl}, which is an analogous of the classical bt--algebra in the context of Coxeter groups of type $\mathtt{B}$.

\begin{definition}\label{bt-typeB}
  Let $n\geq 2$. We define a bt--algebra of type $\B$, denoted by $\E=\E(\U,\V)$, as the algebra generated by $B_1, T_1\dots, T_{n-1}$ and $F_1,\dots F_n, E_1\dots, E_{n-1}$, subject to the following relations
 \begin{eqnarray}
    T_iT_j &=& T_jT_i\quad \mbox{for all $|i-j|>1$}, \label{first}\\
    T_iT_{i+1}T_i &=&T_{i+1}T_iT_{i+1} \quad \mbox{for all $1\leq i\leq n-2$}, \\
    T_i^2&=&1+(\U-\U^{-1})E_iT_i \quad \mbox{for all $1\leq i\leq n-1$}, \\
    E_i^2 &=& E_i  \quad \mbox{for all $i$},\\
    E_iE_j&=&E_jE_i  \quad \mbox{for all $i, j$}\\
    E_iT_i&=&T_iE_i   \quad \mbox{for all $1\leq i\leq n-1$},\\
    E_iT_j&=&T_jE_i \quad \mbox{for all $|i-j|>1$}, \\
    E_iE_jT_i&=&T_iE_iE_j= E_j T_i E_j \quad \mbox{for all $|i-j|=1$},\\
    E_iT_jT_i&=&T_jT_iE_j\quad \mbox{for all $|i-j|=1$},\label{lastA}\\
    B_1T_1B_1T_1&=&T_1B_1T_1B_1,\label{firstB}  \\
    B_1T_i&=&T_iB_1 \quad \mbox{for all $i>1$},\\
    B_1^2&=&1+(\V-\V^{-1})F_1B_1,\label{quadratictipoB} \\
    B_1E_i&=&E_iB_1\quad \mbox{for all $i$}, \label{last}\\
    F_i^2&=&F_i \quad \mbox{for all $i$}, \label{idempotentf} \\
    B_1F_j&=&F_jB_1 \quad \mbox{for all $j$}, \label{actb}\\
    F_iE_j&=&E_jF_i \quad \mbox{for all $i,j$},\\
    F_jT_i&=&T_iF_{s_i(j)} \label{perm}\quad \text{where $s_i=(i,i+1)$},\\
    E_iF_i&=&F_iF_{i+1}=E_iF_{i+1} \label{lastB}\quad \mbox{for all $1\leq i\leq n-1$}.
 \end{eqnarray}
 For $n=1$, we define the algebra $\mathcal{E}_1^{\B}$ as the algebra generated by $1, B_1$ and $F_1$ subject to the relations (\ref{quadratictipoB}), (\ref{idempotentf}) and (\ref{actb}).
\end{definition}

\begin{proposition}\label{rep}
  The mapping $\sigma_i\mapsto T_i$, $\rho_1\mapsto B_1$, $\eta_i\mapsto E_i$ and $\phi_1\mapsto F_1$ defines a representation from $\TB$ into $\E$, denoted by $\theta$.
\end{proposition}
\begin{proof}
  It is enough to prove that $T_i$, $B_1$, $E_i$ and $F_1$, for $1\leq i \leq n-1$ satisfy the defining relations of $\TB$. %That is, the  hold when we replace $\sigma_i$, $\rho_1$, $\eta_i$ and $\phi_j$ by in the.
  By \cite[Proposition 4.2]{aiju2}, we have that the relations of type $\mathtt{A}$ are satisifed by the elements $T_i$'s and $E_i$'s. On the other hand, using the defining relations (\ref{firstB})--(\ref{lastB}) of $\E$, we obtain that the generators of $\E$ also satisfy type $\mathtt{B}$ relations.

\end{proof}
We now recall the definition of the Markov trace supported by the algebra $\E$, which is the main ingredient of the Jones method.
\begin{theorem}\label{Markovtrace}
  $\mathtt{tr}$ is a Markov trace on $\{\E\}_{n\geq 1}$. That is, for all $n\geq 1$, the linear map $\mathtt{tr}_n:\E\rightarrow \mathbb{K}$ satisfies the following properties:

\begin{itemize}
  \item[(i)] $\mathtt{tr}_{n}(1)  = 1$
  \item[(ii)] $\mathtt{tr}_{n+1}(XT_{n}) = \mathtt{tr}_{n+1}(XE_{n}T_n)=  \z \mathtt{tr}_{n}(X) $
  \item[(iii)] $\mathtt{tr}_{n+1}(XE_{n}) = \mathtt{tr}_{n+1}(XF_{n+1}) =\x \mathtt{tr}_{n}(X)$
  \item[(iv)] $\mathtt{tr}_{n+1}(XB_{n}) =\y \mathtt{tr}_{n}(X)$
  \item[(v)] $\mathtt{tr}_{n+1}(XB_{n+1}E_n) =\mathtt{tr}_{n+1}(XB_{n+1}F_{n+1})=\w \mathtt{tr}_{n}(X)$
  \item[(vi)] $\mathtt{tr}_n(XY)  =  \mathtt{tr_n}(YX), $

\end{itemize}
where $X,Y\in \E$.
\end{theorem}

In \cite[Section 6]{fl}, we define the invariant $\overline{\Delta}_{\mathtt{B}}$ for classical links in the solid torus, by using Jones method. This, invariant is essentially the composition of $\pi$, the natural representation of $W_n$ into $\E$, and the Markov trace from Theorem~\ref{Markovtrace} (up to normalization and re--escalation). Analogously, we now construct an extension of such invariant, which is also denoted by $\overline{\Delta}_{\mathtt{B}}$, to simplify notation.

Set
\begin{equation}\label{CapitalLambda}
\mathsf{L} := \frac{\z - ({\U}- {\U}^{-1})\x}{\z} \quad \text{and} \quad D:=\frac{1}{\z \sqrt{\mathsf{L}}}.
\end{equation}
Let $\theta_{\mathsf{L}}$ be the representation of $\TB$ in $\E$, defined by the mapping $\sigma_i \mapsto \sqrt{\mathsf{L}}T_i$, $\rho_1\mapsto B_1$, $\eta_i\mapsto E_i$ and $\phi_1\mapsto F_1$. Then, for $\alpha\in \TB $, we define
\begin{equation}\label{inv1}
\overline{\Delta}_{\mathtt{B}}(\alpha):=(D)^{n-1}(\mathtt{tr}_n\circ \theta_{\mathsf{L}})(\alpha).
\end{equation}
It is well know that the previous expression can be rewritten as follows
\begin{equation}\label{inv2}
\overline{\Delta}_{\mathtt{B}}(\alpha)=(D)^{n-1}(\sqrt{\mathsf{L}})^{e(\alpha)}(\mathtt{tr}_n\circ \theta)(\alpha),
\end{equation}
where $e(\alpha)$ is the exponent sum of the $\sigma_i$'s appearing in the braid $\alpha$, and $\theta$ is the representation from Proposition~\ref{rep}. Similarly to \cite[Theorem 4]{fl}, we obtain that $\overline{\Delta}_{\mathtt{B}}$ is an invariant for tied links in $ST$. Moreover, we have the following result.

\begin{theorem}
  Let $L$ be a tied link in $\mathcal{T}_{ST}$ obtained by closing a braid $\alpha\in \TB$. Then, we have $\overline{\Delta}_{\mathtt{B}}(\alpha)= \mathcal{F}_{\B}(L)$.
\end{theorem}
\begin{proof}
 It is enough to prove that the invariant $\overline{\Delta}_{\mathtt{B}}$ satisfies the skein relations of $\mathcal{F}_{\B}$ (see Theorem~\ref{skeinth}). Firstly, note that the unknots $\bigcirc, \tilde{\bigcirc}, \bigcirc_+$ $\tilde{\bigcirc}_+$ correspond to $\widehat{1}$, $\widehat{\phi_1}$, $\widehat{\rho_1}$ and  $\widehat{\rho_1\phi_1}$, respectively. Thus, by trace conditions, we have that $\overline{\Delta}_{\mathtt{B}}(\alpha)=\mathcal{F}_{\B}(\widehat{\alpha})$ for all $\alpha\in TB_1^{\mathtt B} $, that is, $\overline{\Delta}_{\mathtt{B}}$ satisfies the initial conditions (i) from Theorem~\ref{skeinth}.\smallbreak
 Let $\alpha$ be a tied braid in $\TB$, and set $L=\widehat{\alpha}$. Let $\alpha_+$, $\alpha_-$ and $\alpha_{\sim}$ be the tied braids that are identical outside the small disk, whereas inside the disk look according to Figure~\ref{skein1}. Then, we have that $L_+=\alpha_+$, $L_-=\alpha_-$ and $L_{\sim}=\alpha_{\sim}$.
 \smallbreak
 Therefore, using the quadratic relation of $\E$, we have that
  \begin{equation}\label{relaciontraza}
    \tr(\theta(\alpha_+))=\tr(\theta(\alpha_-))+(\U-\U^{-1})\tr(\theta(\alpha_{sim})).
  \end{equation}
Note now that $e(\alpha_{\sim})=e(\alpha_{+})-1=e(\alpha_{-})+1$. Thus, we have
\begin{align*}
  \overline{\Delta}_{\mathtt{B}}(\alpha_+)=& D^{n-1}\lambda^{e(\alpha_{\sim})+1}\tr(\theta(\alpha_+)), \quad \overline{\Delta}_{\mathtt{B}}(\alpha_-)=D^{n-1}\lambda^{e(\alpha_{\sim})-1}\tr(\theta(\alpha_-))\ \mbox{\ and}  \\
    \overline{\Delta}_{\mathtt{B}}(\alpha_{\sim})=&D^{n-1}\lambda^{e(\alpha_{\sim})}\tr(\theta(\alpha_{\sim}))
\end{align*}
Therefore, from Eq. (\ref{relaciontraza}) we obtain
$$\frac{1}{\lambda}\overline{\Delta}_{\mathtt{B}}(\alpha_+)-\lambda\overline{\Delta}_{\mathtt{B}}(\alpha_-)=(\U-\U^{-1})\overline{\Delta}_{\mathtt{B}}\alpha_{\sim}),$$
as we wanted. Finally, we can prove analogously that $\overline{\Delta}_{\mathtt{B}}(L_+)$ satisfies the second skein relation by using the quadratic relation (\ref{quadratictipoB}) of $\E$.
 \end{proof}

\section*{Acknowledgements.} The author wishes to thank the Math section of ICTP, where the paper was written, for the invitation and hospitality. In particular, the author wishes to express his gratitude to Francesca Aicardi for their helpful comments, which were essential to develop this work.

 \bibliography{TLinST}{}
\bibliographystyle{acm}

\end{document}